\documentclass[11pt,a4paper]{article}

\usepackage{amsmath,amsthm,amsfonts}
\usepackage{graphicx}
\usepackage{color}
\usepackage{paralist}

\usepackage{hyperref}
\hypersetup{%
plainpages=false,
colorlinks,urlcolor=blue,linkcolor=blue,citecolor=blue,
pdftitle={Electricity Capacity Assessment:
  analysis and integration of the demand-wind relationship},
pdfauthor={Stan Zachary},
pdfsubject={},
pdfpagemode={UseOutlines},
pdfpagelayout={OneColumn},
pdfstartview={Dubhe}
}

\newtheorem{proposition}{Proposition}
\theoremstyle{remark}
\newtheorem{remark}{Remark}

\renewcommand{\S}{\mathcal{S}}
\newcommand{\D}{\mathcal{D}}

\newcommand{\lole}{\ensuremath{\mathrm{LOLE}}}
\newcommand{\eeu}{\ensuremath{\mathrm{EEU}}}
\newcommand{\cone}{\ensuremath{\mathrm{CONE}}}
\newcommand{\voll}{\ensuremath{\mathrm{VOLL}}}
\newcommand{\ctp}{\textit{capacity to procure}}
\newcommand{\mmx}[1]{\ensuremath{#1^*}}
\newcommand{\lwr}[1]{\ensuremath{\hat{#1}}}

\parskip\smallskipamount
\parskip \smallskipamount

\title{Least worst regret analysis for decision making under
  uncertainty,\\ with applications to future energy scenarios}
\author{Stan Zachary\footnote{Research
    supported by EPSRC grant EP/I017054/1.}\\[1 ex]
  Heriot-Watt University}
\date{\today}

\begin{document}

\maketitle

\begin{abstract}

  Least worst regret (and sometimes minimax) analysis are often used
  for decision making whenever it is difficult, or inappropriate, to
  attach probabilities to possible future scenarios.  We show that,
  for each of these two approaches and subject only to the convexity
  of the cost functions involved, it is always the case that there
  exist two ``extreme'' scenarios whose costs determine the outcome of
  the analysis in the sense we make clear.  The results of either
  analysis are therefore particularly sensitive to the cost functions
  associated with these two scenarios, while being largely unaffected
  by those associated with the remainder.  Great care is therefore
  required in applications to identify these scenarios and to consider
  their reasonableness.

  We also consider the relationship between the outcome of a least
  worst regret and a Bayesian analysis, particularly in the case where
  the regret functions associated with the scenarios largely
  differ from each other by shifts in their arguments, as is the case
  in many applications.

  We study in detail the problem of determining an appropriate level
  of electricity capacity procurement in Great Britain, where
  decisions must be made several years in advance, in spite of
  considerable uncertainty as to which of a number of future scenarios
  may occur, and where least worst regret analysis is currently used
  as the basis of decision making.

\end{abstract}

\section{Introduction}
\label{sec:introduction}

Many economic decisions, for example that of an appropriate level of
investment, need to be taken in the face of uncertainties about future
conditions.  Were these conditions known an optimal decision might be
made, usually on the basis of the minimisation of a cost function.
Indeed were it possible to assign probabilities to future
conditions---frequently condensed into a finite number of
\emph{scenarios}---one might reasonably minimise an expected cost, or
perhaps some other appropriate functional of the probability
distribution of cost (see, for example, Berger~\cite{Ber}).  However,
even this is frequently not possible, either because of insufficient
data or other information with which to make an appropriate assessment
of probabilities, or because future conditions are dependent upon
other human decisions---for example political decisions---yet to be
made, and it would be inappropriate to attempt to second-guess these.

A Bayesian decision maker would simply assign subjective probabilities
to future scenarios, using such information as \emph{was} available,
and then proceed as above.  However, there is often an understandable
reluctance to do this, and then recourse is made to one of various
economic decision making criteria which do not require the assignment
of probabilities, and which are therefore often described as
``robust''.  Whether this really makes sense is a matter of
philosophical debate: if one believes that the probabilities
matter---at least, for example, that highly likely future scenarios
should have more weight in the decision making process than highly
unlikely ones---then the use of techniques which entirely ignore
probabilities might simply be regarded as suboptimal.  However, the
continued use of such techniques is often justified on pragmatic
grounds.

In the present paper, we study two commonly used non-probabilistic
techniques.  In both cases associated with each future scenario to be
considered is a cost function defined on the set of possible
decisions.  In \emph{minimax} analysis that decision is made which
minimises the maximum cost to be incurred over all possible scenarios.
This technique, based on extreme risk aversion, has a long history;
for an early systematic treatment, see Savage~\cite{Sav}.  In
\emph{least worst regret} (or \emph{minimax regret}) analysis, the
cost function associated with each scenario is modified by subtracting
its minimum value over the decision set; the resulting functions are
termed the \emph{regret functions} associated with the scenarios, and
the minimax criterion is applied to the set of regret functions rather
than the original cost functions.  This technique was first introduced
in the early 1980s, independently in the papers by Loomes and
Sugden~\cite{LS} and by Bell~\cite{Bell} and in the book by
Fishburn~\cite{Fish}, and has been the subject of some subsequent
study, usually in the context of specific and often complex models.

More formal definitions of both these techniques are given in the next
section.  Least worst regret analysis in particular is much used in
practice. However, depending on the nature of the cost functions
involved, neither technique can be said to be always scientifically
rational.  \emph{Minimax} analysis suffers from the problem that a
single scenario with an associated cost function which is uniformly
high across the decision set (notably one which is pointwise higher
than the cost functions associated with all other scenarios) may alone
determine the outcome of the analysis, even though this scenario may
be unlikely and/or may have a relatively flat cost function---note
that either of the latter circumstances strongly suggests that such a
scenario should \emph{not} be influential in cost-based decision
making.  \emph{Least worst regret} analysis avoids this particular
problem by effectively adjusting the cost function associated with
each scenario so that its minimum value across the decision set is
zero (thus defining the regret function).  There are some
well-rehearsed arguments as to why this is to be preferred---see, for
example, the references cited above, and note also that a Bayesian
probabilistic analysis also depends on the cost functions only through
their associated regret functions, as noted more formally in the next
section.  However least worst regret analysis has the unfortunate
property that, depending on the cost functions involved, which of two
decisions is to be preferred may depend on the costs associated with a
third possible decision----or indeed on whether that third possible
decision is actually included in the decision set---a point to which
we return in more detail in Section~\ref{sec:least-worst-regret}.
Neither minimax analysis nor a Bayesian probabilistic analysis suffers
from this problem.

Thus we would argue that the \emph{uncritical} adoption of either of
the above techniques is to be avoided.  Whether either of them works
well in practice depends very much on the shape of the cost functions
involved, as well as on the specification of the decision set itself.
This is one issue we wish to explore in the present paper.  Our
particular interest is in least worst regret analysis, as the latter
seems to be presently widely used.  One such use is that of the
determination, in Great Britain, of an optimal level of provision of
``conventional'' electricity generation capacity; this is something
which must be decided several years in advance, even although there
are considerable uncertainties as to both future electricity demand
and the future availability of ``non-conventional'', i.e.\ renewable,
generation.  We study this example in
Section~\ref{sec:exampl-electr-capac}.

In Section~\ref{sec:least-worst-regret} we give more formal
definitions and study some relevant mathematical properties of both
minimax and least worst regret analysis.  Our main result is that in
either case, given the convexity of the cost functions, it is possible
to identify two ``extreme'' scenarios which essentially determine the
outcome of the analysis in a sense we make precise there---roughly
speaking the other scenarios do not matter at all provided none of
them becomes more extreme than the two above.  It follows that great
care needs to be taken in the specification of the most ``extreme''
scenarios to be considered.  We also study the common situation in
which, at least to a first approximation, the differences between the
scenario regret functions are represented by shifts in their
arguments; here the ``extreme'' scenarios are immediately identifiable
and the properties of a least worst regret analysis---particularly
robustness with respect to scenario variation---readily understood.
Finally in that section, we consider briefly the alternative Bayesian
analysis and how the outcome of a least worst regret analysis might
relate to it.

Section~\ref{sec:exampl-investm-reduc} studies the problem in which
each of the scenario cost functions is---again at least
approximately---the sum of an exponentially decaying and a linearly
increasing term. Here it is possible to make quite quantitative
deductions about the outcome of a least worst regret analysis and its
relation to that of possible Bayesian analyses.  This situation is
common when the decision problem is that of choosing an appropriate
level of investment---in the face of future uncertainty---so as to
manage risk.  It is in particular the case (again to a good
approximation) for the Great Britain electricity capacity procurement
problem which was introduced above and which is studied in detail in
the following section.

In the concluding Section~\ref{sec:conclusion} we make some further
brief comments about the sensible application of least worst regret
analysis.

\section{Minimax and least worst regret analysis}
\label{sec:least-worst-regret}

Both \emph{minimax} and \emph{least worst regret} (LWR) analysis are
defined in relation to a set~$\S$ of \emph{scenarios}---which in the
present note we take to be finite---and a set~$\D$ of possible
\emph{decisions}.  The decision set~$\D$ may be finite or infinite.
Corresponding to each scenario~$i\in\S$ is a \emph{cost
  function}~$f_i$ defined on the set~$\D$ such that for each $x\in\D$
the quantity~$f_i(x)$ is the cost associated with the decision~$x$
under the scenario~$i$.  The decision problem is that of choosing
$x\in\D$ such that the associated costs $f_i(x)$ are in some sense
minimised.  This problem as stated is not yet well-defined, and our
interest is the mathematical properties and relationships between the
various rival approaches to making it well-defined, in particular with
the \emph{minimax}, \emph{least worst regret} and \emph{Bayesian}
approaches (where the latter additionally requires an assignment of
probabilities to scenarios).

We shall find it convenient to think also of the decision problem
associated with any subset~$\S'$ of the set of scenarios~$\S$.  For
any such subset $\S'$, define further the maximum cost
function~$f_{\S'}$ by
\begin{equation}
  \label{eq:1}
  f_{\S'}(x) = \max_{i\in\S'}f_i(x).
\end{equation}
Then the \emph{minimax} solution $\mmx{x}_{\S'}$ to the decision
problem associated with the set~$\S'$ is the value of $x$ within the
decision set~$\D$ which minimises~$f_{\S'}(x)$.  In particular
$\mmx{x}_{\S}$ is the minimax solution associated with the entire set
of scenarios~$\S$.

We shall find it convenient to write also $\mmx{x}_i=\mmx{x}_{\{i\}}$
for each $i\in\S$ and similarly to write
$\mmx{x}_{kl}=\mmx{x}_{\{k,l\}}$ for each subset $\{k,l\}$ of $\S$ of
size two.  For each $i\in\S$, define also the \emph{regret}
function~$\hat{f}_i$ by
\begin{equation}\label{eq:2}
  \hat{f}_i(x) = f_i(x) - f_i(\mmx{x}_i)
\end{equation}
(where, as defined above, $\mmx{x}_i$ is the value of $x$ which
minimises $f_i(x)$ in the decision set~$\D$).  Analogously to the
definition of \eqref{eq:1}, for any subset $\S'$ of $S$, define also
the \emph{worst regret} function~$f_{\S'}$ associated with the set of
scenarios~$\S'$ by
\begin{equation}
  \label{eq:3}
  \hat{f}_{\S'}(x) = \max_{i\in\S'}\hat{f}_i(x).
\end{equation}
Then the \emph{least worst regret} (LWR) (or \emph{minimax regret})
solution $\lwr{x}_{\S'}$ to the decision problem associated with the
set~$\S'$ is the value of $x$ which minimises~$\hat{f}_{\S'}(x)$.
That is, the LWR solution $\lwr{x}_{\S'}$ associated with the
set~$\S'$ is the minimax solution associated with that set in the case
where the original cost functions~$f_i$, $i\in\S$, are replaced by the
regret functions~$\hat{f}_i$, $i\in\S$.  In particular $\lwr{x}_{\S}$
is the LWR solution associated with the entire set of scenarios~$\S$.

Note that while $\lwr{x}_i=\mmx{x}_i$ for each $i\in\S$, for any
subset~$\S'$ of $S$ of size at least two the quantities~$\mmx{x}_{S'}$
and $\lwr{x}_{S'}$ are in general different.  Further, if any of the
functions~$f_i$, $i\in\S'$, is adjusted by the addition (or
subtraction) of a constant, then the minimax solution~$\mmx{x}_{\S'}$
will in general change while (since the regret function $\hat{f}_i$
will be unaffected) the LWR solution $\lwr{x}_{\S'}$ will remain
unchanged.

However, note also that LWR analysis has the following dubious
property, which is essentially a restatement of that referred to in
the Introduction.  Suppose that the decision set~$\D$ is restricted to
some subset~$\D'$, and suppose further that in consequence there is at
least one scenario~$i\in\S$ such that the minimising value $\mmx{x}_i$
of the cost function~$f_i$ within the set~$\D$ fails to belong to the
set~$\D'$.  Then the regret function $\hat{f}_i$ defined in relation
to the set~$\D'$ differs from that defined in relation to the
set~$\D$.  Thus, for any set $\S'$ of scenarios, the restriction of
the decision set~$\D$ to $\D'$ in general changes the
solution~$\lwr{x}_{\S'}$ of the LWR analysis---even although the
original solution (as well as the new solution) may itself belong to
both $\D'$ and $\D$.  We therefore have the situation where which of
two possible decisions is to be preferred depends---somewhat
illogically---on what is happening elsewhere in the decision set.

We are interested in some further properties of both minimax and LWR
analysis of the decision problems defined by the possible scenario
sets~$\S'$, and in particular by the ``maximum'' set of
scenarios~$\S$.  Since as, explained above, for any such~$\S'$ the LWR
solution~$\lwr{x}_{\S'}$ is the minimax solution applied to the regret
functions~$\hat{f}_i$, $i\in\S'$, we concentrate initially on minimax
analysis.  Our results then transfer easily to LWR analysis.

\subsection{Minimax analysis}
\label{sec:minimax-analysis}

It is possible, but unusual in applications, that there exists some
scenario $k\in\S$ such that
\begin{equation}
  \label{eq:4}
  f_i(\mmx{x}_k) \le f_k(\mmx{x}_k) \quad\text{for all $i\in\S$}.
\end{equation}
It then follows immediately that $\mmx{x}_\S=\mmx{x}_k$ (and of course
that $\mmx{x}_{\S'}=\mmx{x}_k$ for any subset~$\S'$ of $\S$ such that
$k\in\S'$).  In this case we may think of the minimax
solution~$\mmx{x}_\S$ as being ``determined'' by the single
scenario~$k$ in the sense that if the functions $f_i$ associated with
the remaining scenarios~$i\neq k$ are varied within the region in
which~\eqref{eq:4} continues to hold, then we still have
$\mmx{x}_\S=\mmx{x}_k$.

We now assume that $|S|\ge2$, i.e.\ that there are at least two
scenarios.  Then, although there may be no single scenario~$k\in\S$
such that the relation~\eqref{eq:4} holds, it does very often
remain the case that we can find two scenarios $k,l\in\S$ such that
\begin{equation}
  \label{eq:5}
  f_i(\mmx{x}_{kl}) \le f_{kl}(\mmx{x}_{kl}) \quad\text{for all $i\in\S$}.
\end{equation}
Analogously to the earlier situation, it then follows that
$\mmx{x}_\S=\mmx{x}_{kl}$ and we may think of the minimax
solution~$\mmx{x}_\S$ as being ``determined'' by the two scenarios~$k$
and $l$ in the same sense as previously, i.e.\ if the functions $f_i$
associated with the remaining scenarios~$i\ne k$, $i\ne l$, are varied
within the region in which~\eqref{eq:5} continues to hold, then still
we have $\mmx{x}_\S=\mmx{x}_{kl}$.

Note that if the relation~\eqref{eq:4} \emph{does} hold for some
$k\in\S$, then the relation~\eqref{eq:5} also holds for that~$k$ and
for any other $l\in\S$ (with $\mmx{x}_{S}=\mmx{x}_{kl}=\mmx{x}_k$), so
that the former situation may be regarded as a special case of the
latter.

Figure~\ref{fig:lwr} illustrates the typical situation in which the
relation~\eqref{eq:5} holds: the decision set~$\D$ is taken to be the
entire real line and the cost functions~$f_i$ corresponding to
five scenarios~$i\in\S$ are plotted; it is seen that~\eqref{eq:5}
holds with $k$ and $l$ given (uniquely) by the two scenarios whose
cost functions are shown as solid lines, with the remaining cost
functions being shown as dashed lines.  As above, we then have that
the minimax solution~$\mmx{x}_\S=\mmx{x}_{kl}$ is ``determined'' by
the two scenarios~$k$ and $l$ in the sense discussed above.

\begin{figure}[!ht]
  \centering
  \includegraphics[scale=0.8]{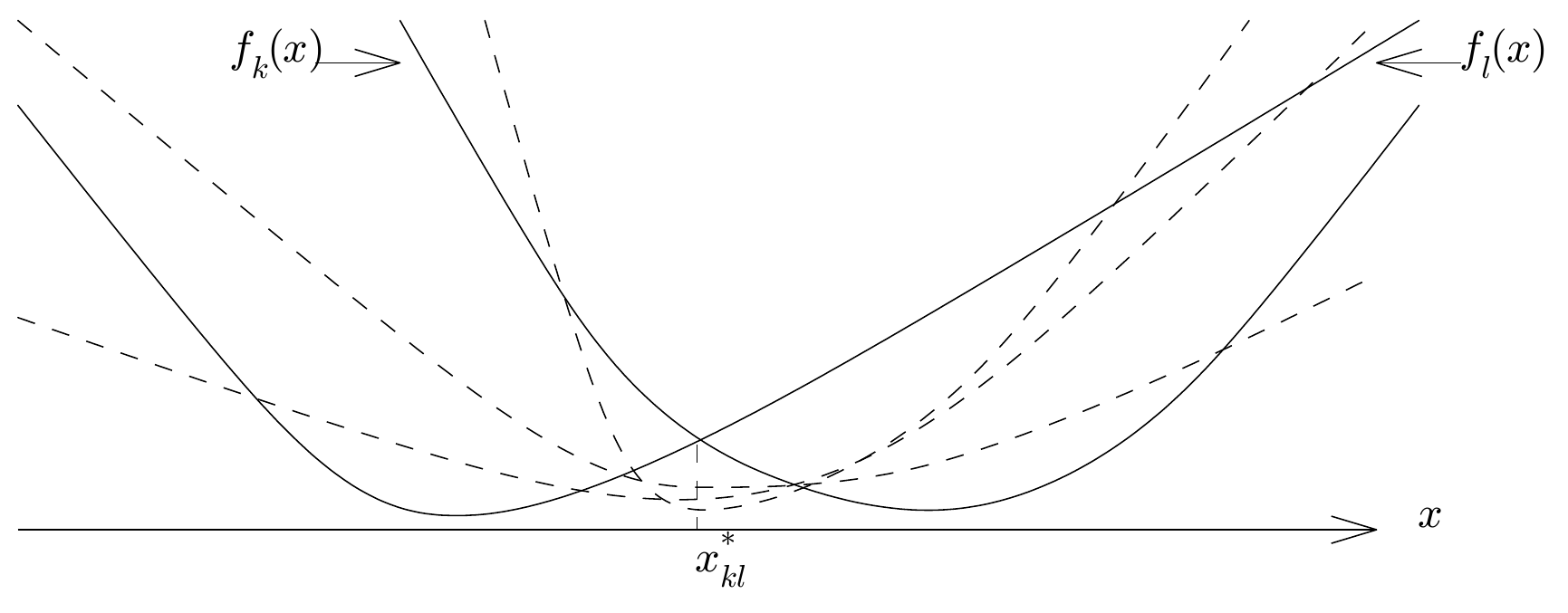}
  \caption{Illustrative cost functions for a set of five scenarios:
    the relation~\eqref{eq:5} holds for the two scenarios~$k$ and $l$
    whose cost functions are shown as solid lines.}
  \label{fig:lwr}
\end{figure}

We now specialise to the case where the decision~$\D$ is some interval
$[x_a,x_b]$ of the real line, where we may have $x_a=-\infty$ and/or
$x_b=\infty$, and where (as in the example of Figure~\ref{fig:lwr})
the functions~$f_i$, $i\in\S$, are convex.  These conditions are
usually satisfied in practice.  For simplicity we further assume that
the functions~$f_i$ are strictly convex, so that in particular all
minima below will be uniquely defined; we subsequently remark how the
assumption of strict convexity may be relaxed.  We then have the
following result.

\begin{proposition}\label{proposition:1}
  Assume that the decision set~$\D$ is given by some interval
  $[x_a,x_b]$ of the real line and that the functions~$f_i$ are
  strictly convex.  Then there always exist two scenarios $k,l\in\S$
  such that the relation~\eqref{eq:5} holds.  In particular we then
  have, as above, that the minimax solution $\mmx{x}_\S$ associated
  with the set~$\S$ is equal to the minimax solution $\mmx{x}_{kl}$
  defined by the scenarios~$k$ and $l$ (and indeed
  $\mmx{x}_{\S'}=\mmx{x}_{kl}$ for any subset~$\S'$ of $\S$ which
  contains both $k$ and $l$).
\end{proposition}

\begin{proof}
  Suppose first that $x_a$ and $x_b$ are finite and that the
  functions~$f_i$ are further differentiable.  Clearly, there exists
  some $k\in\S$ such that $f_k(\mmx{x}_\S)=f_\S(\mmx{x}_\S)$.  Note
  that this implies that
  \begin{equation}
    \label{eq:6}
    f_i(\mmx{x}_\S) \le f_k(\mmx{x}_\S)
    \qquad\text{for all $i\in\S$.}
  \end{equation}
  Now observe that if $f_k'(\mmx{x}_\S)\ge0$ then, by convexity,
  $f_k(x)>f_k(\mmx{x}_\S)$ for all $x\in(\mmx{x}_\S,x_b]$ and hence
  $\mmx{x}_k\le \mmx{x}_\S$.  Similarly if $f_k'(\mmx{x}_\S)\le0$ then
  $f_k(x)>f_k(\mmx{x}_\S)$ for all $x\in[x_a,\mmx{x}_\S)$ and hence
  $\mmx{x}_k\ge \mmx{x}_\S$.  It follows that if
  $f_k'(\mmx{x}_\S)\ge0$ and $\mmx{x}_\S=x_a$, or if
  $f_k'(\mmx{x}_\S)\le0$ and $\mmx{x}_\S=x_b$, or if
  $f_k'(\mmx{x}_\S)=0$, then $\mmx{x}_k=\mmx{x}_\S$; if this is so,
  then it further follows from~\eqref{eq:6} that the
  relation~\eqref{eq:4} holds and so, as already noted, the
  relation~\eqref{eq:5} holds for~$k$ and any $l\ne k$.

  Otherwise, it is the case that either $f_k'(\mmx{x}_\S)>0$ and
  $\mmx{x}_\S>x_a$ or else $f_k'(\mmx{x}_\S)<0$ and $\mmx{x}_\S<x_b$.
  Without loss of generality, we suppose the first of these two
  possibilities to be the case.  It follows from these conditions and
  from~\eqref{eq:6} that there necessarily exists some $l\ne k$ such
  that $f_l(\mmx{x}_\S)=f_k(\mmx{x}_\S)$ and $f_l'(\mmx{x}_\S)\le0$
  (for otherwise there would exist some $x'\in[x_a,\mmx{x}_\S)$ such
  that $f_\S(x')<f_k(\mmx{x}_\S)=f_\S(\mmx{x}_\S)$, in contradiction
  to the definition of $\mmx{x}_\S$).  Then, from convexity as above,
  \begin{equation}
    \label{eq:7}
    f_l(x)>f_l(\mmx{x}_\S)
    \qquad\text{for all $x\in[x_a,\mmx{x}_\S)$.}
  \end{equation}
  Since also (once more by convexity)
  \begin{equation}
    \label{eq:8}
    f_k(x)>f_k(\mmx{x}_\S)=f_l(\mmx{x}_\S)
    \qquad\text{for all $x\in(\mmx{x}_\S,x_b]$,}
  \end{equation}
  it follows from \eqref{eq:7} and \eqref{eq:8} that
  $\mmx{x}_{kl}=\mmx{x}_\S$.  Thus (using also~\eqref{eq:6}) the
  relation~\eqref{eq:5} holds.


  Once the above argument is understood, only obvious modifications
  are required to deal with the case where we may have $x_a=-\infty$
  or $x_b=\infty$---e.g.\ we may compactify $\D$ by the addition of
  these points as necessary and extend the definitions of the convex
  functions~$f_i$ and their derivatives by taking limits.  Similarly
  only obvious modifications are required to deal with the case where
  some or all of the convex functions~$f_i$, $i\in\S$, fail to be
  differentiable: any derivative such as $f_i'(x)$ may be replaced by
  the slope of a supporting hyperplane to $f_i$ at~$x$; all that is
  further required is some extra care in dealing with the
  nonuniqueness of these supporting hyperplanes.  Alternatively
  nondifferentiable convex functions may be represented as pointwise
  limits of differentiable convex functions, and standard limiting
  arguments used.
\end{proof}

\begin{remark}
  Under the relaxation of the requirement that the convexity of the
  functions~$f_i$, $i\in\S$, be strict,
  Proposition~\ref{proposition:1} continues to hold provided that, in
  the relation~\eqref{eq:5}, $\mmx{x}_{kl}$ is understood to refer to
  \emph{some} value of the (no longer necessarily uniquely defined)
  minimum of the function $f_{kl}$.
\end{remark}

\subsection{LWR analysis}
\label{sec:lwr-analysis}

As previously remarked, the LWR solution~$\lwr{x}_{\S'}$ associated
with any subset $\S'$ of $\S$ is just the minimax solution applied to
the regret functions $\hat{f}_i$, $i\in\S'$, rather than to the
original cost functions $f_i$, $i\in\S'$.  Therefore the analysis of
Section~\ref{sec:minimax-analysis} is equally applicable to LWR
decision making.

Thus in the case where we can find two scenarios $k,l\in\S$ such that
\begin{equation}
\label{eq:9}
\hat{f}_i(\lwr{x}_{kl}) \le \hat{f}_{kl}(\lwr{x}_{kl})
\quad\text{for all $i\in\S$},
\end{equation}
it follows that $\lwr{x}_\S=\lwr{x}_{kl}$ and we may think of the LWR
solution~$\lwr{x}_\S$ as being ``determined'' by the two scenarios~$k$
and $l$ in the same sense as previously: that is, if the functions
$f_i$ associated with the remaining scenarios~$i\ne k$, $i\ne l$, are
varied within the region in which~\eqref{eq:9} continues to hold,
then still we have $\lwr{x}_\S=\lwr{x}_{kl}$.

We further have that Proposition~\ref{proposition:1} for minimax
analysis translates immediately to the following equivalent result for
LWR analysis.

\begin{proposition}\label{proposition:2}
  Assume that the decision set~$\D$ is given by some interval
  $[x_a,x_b]$ of the real line and that the functions~$f_i$ are
  strictly convex.  Then there always exist two scenarios $k,l\in\S$
  such that the relation~\eqref{eq:9} holds.  In particular we then
  have that the LWR solution $\lwr{x}_\S$ associated with the set~$\S$
  is equal to the LWR solution $\lwr{x}_{kl}$ defined by the
  scenarios~$k$ and $l$ (and indeed $\lwr{x}_{\S'}=\lwr{x}_{kl}$ for
  any subset~$\S'$ of $\S$ which contains both $k$ and $l$).
\end{proposition}

The requirement that the convexity be strict may be relaxed in the
same sense as previously.  Figure~\ref{fig:lwr} continues to
illustrate the situation, if the plotted functions are now regarded as
the regret functions~$\hat{f}_i$ rather than the original cost
functions~$f_i$.

\subsection{Argument shifts.}
\label{sec:linear-shifts}

We continue to concentrate on LWR analysis, and we assume, for
simplicity, that the decision set~$\D$ is the entire real line.
Suppose that the functions~$f_i$---or equivalently the
functions~$\hat{f}_i$---are convex and are ordered by their individual
minimising values~$\lwr{x}_i=\mmx{x}_i$ so that
$\lwr{x}_1\le\dots\le \lwr{x}_n$.  It is then frequently the case that
the two ``extreme'' scenarios~$1$ and $n$ play the role of the
scenarios~$k$ and $l$ such that the relation~\eqref{eq:9} holds, so
that in particular the LWR solution is given by
$\lwr{x}_\S=\lwr{x}_{1n}$ and that yet again this solution is
``determined'' by the scenarios~$1$ and~$n$ in the sense described
earlier.

An important special case in which the above is true occurs when (at
least to a sufficiently good approximation) the functions~$f_i$ are
further such that the regret functions~$\hat{f}_i$ differ simply by
argument shifts, i.e.\ there exist some convex function~$g$ and
constants $a_1<a_2<\dots<a_n$, such that
\begin{equation}
  \label{eq:10}
  \hat{f}_i(x)=g(x-a_i), \qquad i \in \S,
\end{equation}
(implying in particular that $\lwr{x}_i=\lwr{x}+a_i$ for all~$i$, where
$\lwr{x}$ minimises $g(x)$).  In this case we note further that the
value of $\lwr{x}_{1n}$ in relation to $\lwr{x}_1$ and $\lwr{x}_n$ of
depends on shape of the functions~$\hat{f}_1$ and~$\hat{f}_n$ on
either side of their individual minimising values~$\lwr{x}_1$ and
$\lwr{x}_n$.  In particular, if additionally the functions~$\hat{f}_i$
are symmetric about their individual minimising values~$\lwr{x}_i$
then $\lwr{x}_\S=\lwr{x}_{1n}=(\lwr{x}_1+\lwr{x}_n)/2$.  Notably this
symmetry obtains to a sufficiently good approximation when $\lwr{x}_1$
and $\lwr{x}_n$ are sufficiently close ($a_n-a_1$ is sufficiently
small) that these $\hat{f}_1$ and $\hat{f}_n$ may be approximated in
the interval $[\lwr{x}_1,\lwr{x}_n]$ by quadratics with necessarily
the same second derivative.

\subsection{Alternative probabilistic analysis}
\label{sec:altern-prob-analys}

An alternative, essentially Bayesian, analysis may be given by
assigning a probability~$p_i\ge0$ to each possible scenario~$i$ (where
$\sum_{i\in\S}p_i=1$) and then, for example, determining that value
$x_{\mathrm{bayes}}$ of $x$ in the decision set~$\D$ which minimises
the expected cost function
\begin{equation}
  \label{eq:11}
  f(x) = \sum_{i\in\S}p_if_i(x),
\end{equation}
or, equivalently, the expected regret function
$\sum_{i\in\S}p_i\hat{f}_i(x)$.


While decision makers may choose to use LWR analysis so as to avoid
assigning explicit probabilities, it is nevertheless of interest to
understand which sets of probabilities within a Bayesian analysis as
above are compatible with the results of an LWR analysis.  The set of
possible probability measures which may be defined on the $n$
scenarios forms a subset of $(n-1)$-dimensional Euclidean space, and
the requirement the value $x_{\mathrm{bayes}}$ of $x$ which minimises
$f(x)$ as given by~\eqref{eq:11} should be equal to the
LWR~$\lwr{x}_\S$ typically restricts this set of probability measures
to a subset of an $(n-2)$-dimensional space.  We consider the matter
further in Section~\ref{sec:exampl-investm-reduc}.


\section{Investment to reduce risk}
\label{sec:exampl-investm-reduc}

In many applications the decision to be made is that of choosing a
level of provision~$x$, for example for a given industrial
infrastructure, needed to balance associated risk.  The cost of the
provision increases linearly in $x$, while the cost associated with
the risk decays, typically close to exponentially, in $x$.

The example we have in mind particularly is that of electricity
capacity procurement, i.e.\ the determination of the level of
generation capacity which countries (or groups of countries) must make
in order to ensure adequacy of future electricity supplies.  The
distribution over time of the balance of supply over demand for the
future period under study (usually a given year or the peak season of
a year) is modelled by a probability distribution whose left
tail---that corresponding to the region in which supply is
insufficient to meet demand---is typically well approximated by an
exponential function; this distribution is shifted according to the
level of capacity provision~$x$ so that both the expected total
duration of shortfall (the \emph{loss of load expectation}) and the
expected volume of shortfall (the \emph{expected energy unserved}) for
the future period under study decay approximately exponentially in
$x$.  Investment decisions must typically be made some years in
advance and the future risk associated with any given level of
capacity provision depends on which of a number of ``future energy
scenarios'' eventually proves to be most appropriate to the period in
question.  The likelihoods of these scenarios are typically difficult
to quantify probabilistically, and so LWR analysis is often used to
determine the recommended level of investment.

Thus, we have a set of scenarios~$\S$ and, for any scenario~$i$ in
$\S$ and level of capacity provision~$x$, the expected cost of, for
example, a supply-demand imbalance is (at least approximately) of the
form $ a_i e^{-\lambda_i x}$.  If, further, the cost of a level of
investment is proportional to $x$, then the total cost function~$f_i$
associated with each scenario~$i$ is given by
\begin{equation}
  \label{eq:12}
  f_i(x) = b_i e^{-\lambda_i x} + cx,
\end{equation}
for some constants~$\lambda_i>0$, $b_i>0$, and $c>0$, where the latter
typically does not depend on $i$.  For simplicity we again assume the
functions~$f_i$ to be defined on the decision set consisting of the
entire real line---the adjustments required if the decision set is
restricted to some interval of the real line are straightforward.  The
solution~$\lwr{x}_{\S}$ of an LWR analysis is thus readily calculated.
(As already observed this will remain the same under adjustment of any
of the functions~$f_i$ by the addition or subtraction of an arbitrary
constant.)  Since the functions~$f_i$ are strictly convex, it follows
as in Section~\ref{sec:lwr-analysis} that there always exist two
``extreme'' scenarios which ``determine'' the LWR
solution~$\lwr{x}_{\S}$ in the sense discussed there.  In any given
situation this pair is readily numerically identifiable, but our
interest here is to provide some insight as to what it is which really
drives the results of any LWR analysis.

We specialise to the case where the exponential decay
constants~$\lambda_i$ are given by the same constant~$\lambda$ for all
$i\in\S$.  That this is frequently a good approximation in the case of
the electricity capacity procurement example above follows from the
fact that differences between scenarios are often well represented
simply by argument shifts in the distribution of the supply-demand
balance, coupled with the fact that, as already remarked, the left
tail of this distribution is typically exponential---again at least to
a good approximation; it then follows that differences between
scenarios typically correspond simply to differences in the
constants~$b_i$ in~\eqref{eq:12}.  Further, even when this
approximation is less than perfect, the broad conclusions of the analysis
below are likely to continue to apply.  A similar situation occurs in
many other areas of application.

Thus the cost functions $f_i$ are as given by \eqref{eq:12} with
$\lambda_i=\lambda$ for all scenarios~$i$, and we assume that these
scenarios are ordered so that $b_1<\dots<b_n$.  Each function~$f_i$ is
then minimised at
\begin{equation}
  \label{eq:13}
  \lwr{x}_i = \frac{1}{\lambda}\log\frac{b_i\lambda}{c}
\end{equation}
and we have $\lwr{x}_1<\dots<\lwr{x}_n$.  The regret
functions~$\hat{f}_i$ are given by
\begin{equation}
  \label{eq:14}
  \hat{f}_i(x) = b_i e^{-\lambda x} + cx
  - \frac{c}{\lambda}\biggl(1 + \log\frac{b_i\lambda}{c}\biggr)
\end{equation}
for all $x$ as usual, and it is easily seen that these satisfy the
relation~\eqref{eq:10} with the function~$g$ given by
\begin{displaymath}
  g(x) = e^{-\lambda x} + cx
  - \frac{c}{\lambda}\biggl(1 + \log\frac{\lambda}{c}\biggr)
\end{displaymath}
and $a_i=(\log b_i)/\lambda$ for each $i\in\S$.  It follows from the
discussion of Section~\ref{sec:linear-shifts} that the scenarios $1$
and $n$ are ``extreme'' in the sense that the relation~\eqref{eq:9}
holds with $k=1$ and $l=n$, so that once more the LWR solution is
given by $\lwr{x}_\S=\lwr{x}_{1n}$ and this solution is determined by
the ``extreme'' scenarios~$1$ and~$n$ in the sense described in
Section~\ref{sec:least-worst-regret}.  Further the quantity
$\lwr{x}_{1n}$ is the unique solution $x=\lwr{x}_{1n}$ of the equation
$\hat{f}_1(x)=\hat{f}_n(x)$.  Thus an entirely routine calculation
gives that the LWR solution~$\lwr{x}_{\S}$ is given by
\begin{equation}
  \label{eq:15}
  \lwr{x}_{\S}
  = \lwr{x}_{1n}
  = \frac{1}{\lambda}\log\frac{\lambda(b_n-b_1)}{c(\log b_n - \log b_1)},
\end{equation}
irrespective of the values of $b_2,\dots,b_{n-1}$, provided only that
the inequalities $b_1<\dots<b_n$ continue to be satisfied.

Note also that when the extreme scenarios $1$ and $n$ are close to
each other, i.e.\ $b_1$ and $b_n$ are close so that we may write
$b_n=b_1(1+d)$ for some small $d>0$, then again straightforward
calculations give the approximations
\begin{align}
  \lwr{x}_n & \approx \lwr{x}_1 + \frac{d}{\lambda}, \label{eq:16}\\
  \lwr{x}_{1n} & \approx \lwr{x}_1 + \frac{d}{2\lambda}, \label{eq:17}
\end{align}
where in each case the error in the approximation is of the order of
$d^2$ as $d\to0$.  Thus, when $\lwr{x}_1$ and $\lwr{x}_n$ are close,
then $\lwr{x}_{1n}$ is, to a very good approximation, the mean of
$\lwr{x}_1$ and $\lwr{x}_n$.  However, in many applications $d$ as
defined above is not small, and it is easy to check that it is then
the case that $\lwr{x}_{1n}$ is closer to $\lwr{x}_n$ than to
$\lwr{x}_1$.

We now look at the alternative probabilistic, Bayesian, analysis
considered in Section~\ref{sec:altern-prob-analys}, in which a
probability~$p_i\ge0$ (such that $\sum_{i\in\S}p_i=1$) is assigned to
each possible scenario~$i\in\S$.  When the cost functions~$f_i$ are as
given by~\eqref{eq:12}, then the value~$x_{\mathrm{bayes}}$ of $x$
which minimises the (strictly convex) expected cost
function~\eqref{eq:11} is the unique solution $x=x_{\mathrm{bayes}}$
of
\begin{equation}
  \label{eq:18}
  \sum_{i\in\S} p_ib_i\lambda_i e^{-\lambda_i x} = c.
\end{equation}
In particular when $\lambda_i=\lambda$ for all $i$, we have
\begin{equation}
  \label{eq:19}
  x_{\mathrm{bayes}} = \frac{1}{\lambda}\log\frac{\lambda\sum_{i\in\S}p_ib_i}{c},
\end{equation}
and it then follows from \eqref{eq:15} and \eqref{eq:19} that we have 
$\lwr{x}_{\S}=\lwr{x}_{1n}=x_{\mathrm{bayes}}$ if and only if
\begin{equation}
  \label{eq:20}
  \sum_{i\in\S}p_ib_i = \bar{b},
\end{equation}
where the quantity $\bar{b}$ is given by
\begin{equation}
  \label{eq:21}
  \bar{b} = \frac{b_n-b_1}{\log b_n - \log b_1}.
\end{equation}
The equations~\eqref{eq:20} and \eqref{eq:21} therefore determines
those sets of probabilities~$p_i$ which, if used in a probabilistic
analysis, yield the same decision as that given by the LWR analysis.
Note that it is straightforward to show that $b_1<\bar{b}<b_n$---and
indeed that $\bar{b}$ is approximately the mean of $b_1$ and $b_n$
when these two quantities are close to each other.  The
equation~\eqref{eq:20} has the interpretation that, in the case where
$\lambda_i=\lambda$ for all~$i$, the LWR analysis and the
probabilistic analysis yield the same solutions (values of the
decision variable~$x$) if and only if the probabilities~$p_i$ are such
that the corresponding expectation of the scenario parameters $b_i$ is
equal to the ``parameter''~$\bar{b}$. 

\section{Example: electricity capacity procurement in Great Britain}
\label{sec:exampl-electr-capac}

National Grid---the electricity system operator in Great Britain---has
a statutory responsibility to produce an annual report to Government
recommending a level of GB generation capacity procurement in order to
provide adequate security of future electricity supplies.  Its 2015
Electricity Capacity Report (ECR)~\cite{ECR2015} is concerned to
recommend a level of procurement for the ``year'', i.e.\ the winter,
of 2019--20 (winter being the GB peak season of electricity demand).
The report considers a set~$\S$ of 19 possible future scenarios and
``sensitivities'' (we shall sometimes simply refer to all of these as
scenarios) for the above future period.  Associated with each such
scenario~$i$ in $\S$ is an (annual) cost function~$f_i$ given by
\begin{equation}
  \label{eq:22}
  f_i(x) = \voll\times\eeu_i(x) + \cone\times x,
\end{equation}
where $x$ is a possible value of generation \emph{\ctp}---typically
measured in MW---which might be considered for recommendation and
$\eeu_i(x)$ is the corresponding \emph{expected energy unserved}, as
defined in Section~\ref{sec:exampl-investm-reduc} and usually measured
in MWh, over the future period studied; the constants \voll\ and
\cone\ are respectively the \emph{value of lost load} and the
so-called \emph{cost of new entry}, i.e.\ the unit cost per year of
generation capacity which might be procured.  The report uses values
of these constants given by $\voll=\pounds 17,000$/MWh and
$\cone=\pounds 49,000$/MW/year.

The functions $\eeu_i(x)$ decay approximately exponentially in $x$, so
that the cost functions~$f_i$ are strictly convex and approximately of
the form~\eqref{eq:12}; further the exponential constants~$\lambda_i$
in~\eqref{eq:12} are approximately equal over all scenarios $i\in\S$;
however, while the exponential approximation is useful for exploring
many issues, we do not make any formal use of it in the analysis
below.

The choice of recommended \ctp~$x$ is made on the basis of LWR
analysis.  The decision set $\D$ of allowed values of~$x$ to be
considered is not quite continuous, as National Grid also have a
requirement, for the purposes of further analysis, to associate a
particular scenario (or sensitivity) with the value of~$x$ finally
recommended.  The decision set~$\D$ is thus restricted to a set of 19
values $x_i$, $i\in\S$, such that each $x_i$ is just sufficient for
scenario~$i$ to meet a specified reliability standard.  This
reliability standard is fairly closely aligned with the values of the
constants \voll\ and \cone\ so that each value~$x_i$ in the decision
set approximately minimises the function~$f_i(x)$ given by
\eqref{eq:22}.\footnote{The reliability standard requires that each
  $x_i$ should be such that the associated \emph{loss of load
    expectation} (\lole)---again as defined in
  Section~\ref{sec:exampl-investm-reduc}---does not exceed 3 hours per
  year.  Under the exponential approximation~\eqref{eq:12}, each cost
  function~$f_i$ is minimised by that value of~$x$ such that the
  corresponding \lole\ is given by the ratio \cone/\voll, which is
  2.88 hours per year---see~\cite{DECCRS} for the associated analysis.
  This latter figure is sufficiently close to the \lole\ defining the
  reliability standard that, in the case of any \emph{single} scenario
  $i\in\S$, the requirement of choosing~$x$ to meet the reliability
  standard does not significantly conflict with the economic criterion
  of choosing $x$ so as to minimise the cost function~$f_i(x)$.}
However, for the present analysis we treat the decision set~$\D$ as
continuous and formally consisting of all possible values of $x$.  It
turns out that if we do this, and determine the resulting
solution~$\lwr{x}_\S$ of the LWR analysis, the set of allowed valued
considered in the 2015 ECR is sufficiently dense in the neighbourhood
of~$\lwr{x}_\S$ that we can find one such value very close
to~$\lwr{x}_\S$.  Hence the assumption of the present analysis to
treat the decision set~$\D$ as continuous makes very little difference
in practice (see also the results below).

Figure~\ref{fig:ecr_figure_14} relates to the five major scenarios
considered by the 2015 ECR, together with two of the remaining 14
variant scenarios or ``sensitivities'' also considered.  Four of these
five major scenarios were developed by National Grid and are discussed
in detail in the above report.  The fifth is a reference scenario
developed by the UK Department of Energy and Climate Change (DECC).
(Because of a requirement that the recommendations of the report be
independent of Government, this latter scenario is ultimately excluded
from the analysis of the report, something which---as we point out
below---makes no difference to the result of the LWR analysis.)  For
each of the scenarios~$i$ illustrated, Figure~\ref{fig:ecr_figure_14}
plots the total cost~$f_i(x)$ given by \eqref{eq:22} against the
corresponding \emph{\ctp}~$x$, where here $x$ is given in gigawatts.
The curves corresponding to the five major scenarios are shown as
solid lines, while those corresponding to the two variant
``sensitivities'' are shown as dashed lines; the names attached to the
scenarios are those of the 2015 ECR.  This figure is essentially a
reproduction of Figure~14 of the 2015 ECR; the latter figure plots the
same seven cost functions.  Note that these cost functions are all
convex, as are those corresponding to the remaining 12
``sensitivities'' not illustrated.  The latter cost functions are all
pointwise intermediate between those for the two ``sensitivities''
which have been plotted, and are omitted both from the present figure
and from Figure~14 of the 2015 ECR so as to avoid undue clutter (see
also below).

\begin{figure}[!ht]
  \centering
  \includegraphics[scale=0.7]{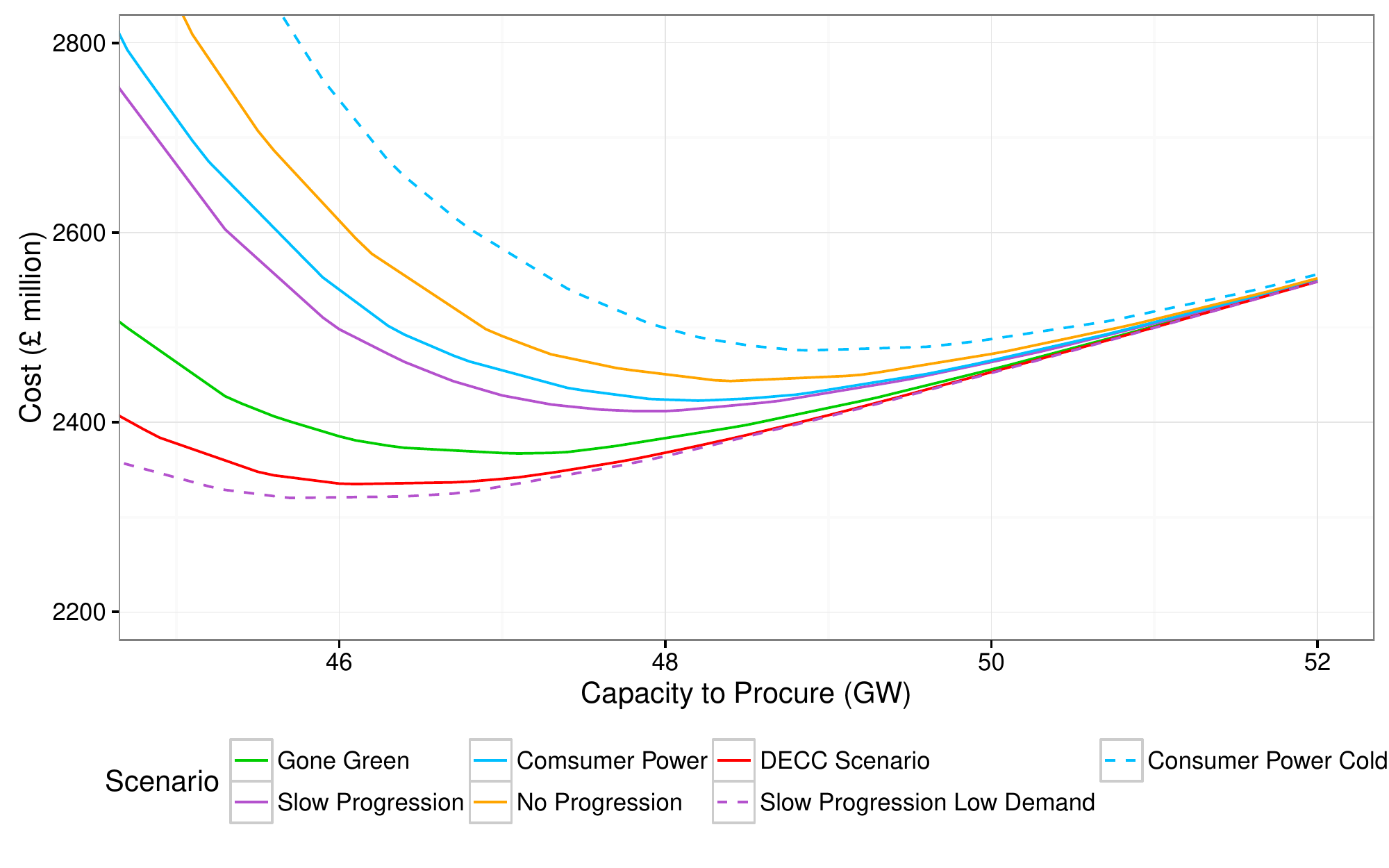}
  \caption{Combined cost of energy unserved and procured capacity,
    i.e.\ the cost functions~$f_i(x)$, against \ctp~$x$ for each of
    the five major scenarios and two variant ``sensitivities''
    illustrated in Figure 14 of the 2015 ECR.}
  \label{fig:ecr_figure_14}
\end{figure}


Figure~\ref{fig:ECR_regret_functions} plots the regret functions
$\hat{f}_i(x)$ against \emph{\ctp}~$x$ for the same seven scenarios
illustrated in Figure~\ref{fig:ecr_figure_14}.  It is now seen that
the two scenarios considered as ``sensitivities'', plotted as dashed
lines, and labelled Slow Progression Low Demand and Consumer Power
Cold, are such that the relation~\eqref{eq:5} is satisfied with $k$
and $l$ indexing these two scenarios, and where the set~$\S$ is here
taken to be the set of seven scenarios illustrated.  That this is so
is unsurprising insofar as the seven regret functions~$\hat{f}_i$,
$i\in\S$, illustrated are clearly---to a good approximation---related
to each other by argument shifts, as described in
Section~\ref{sec:linear-shifts}, and the two ``sensitivities''
identified above are clearly extreme in the sense discussed there.

\begin{figure}[!ht]
  \centering
  \includegraphics[scale=0.7]{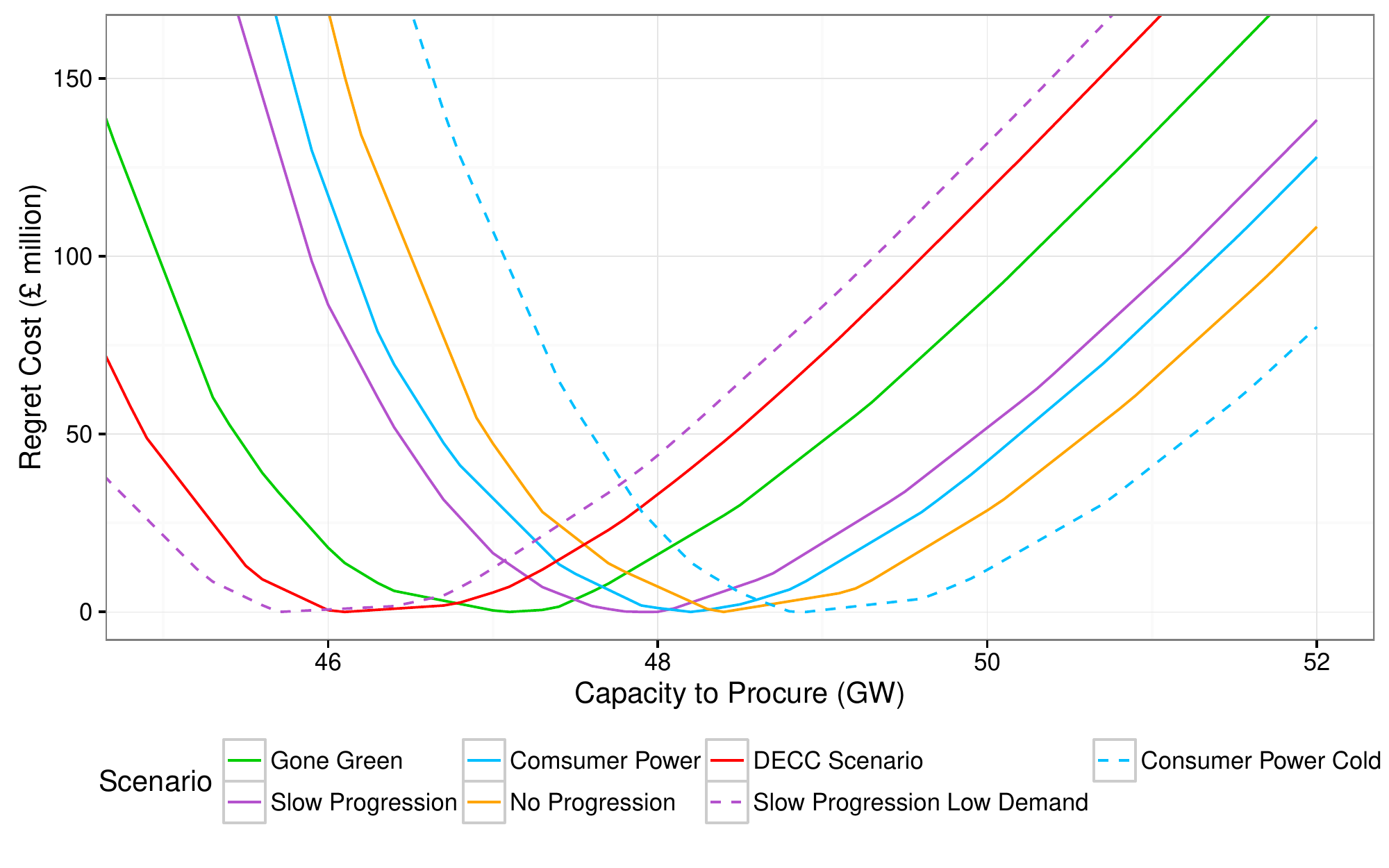}
  \caption{Regret functions~$\hat{f}_i(x)$, against \ctp~$x$ for each
    of the five major scenarios and two variant ``sensitivities''
    illustrated in Figure 14 of the 2015 ECR.}
  \label{fig:ECR_regret_functions}
\end{figure}

We thus have $\lwr{x}_\S=\lwr{x}_{kl}=47.8$ GW---this value
corresponding to the point of intersection of the regret functions for
the above two ``sensitivities''.  All this continues to be the case
when the set~$\S$ is taken to consist of all 19 scenarios considered
by the 2015 ECR; that the relation~\eqref{eq:5} then continues to hold
with $k$ and $l$ indexing the same two ``sensitivities'' identified
above (i.e.\ Slow Progression Low Demand and Consumer Power Cold)
would be immediately apparent if, for example, all the 19 regret
functions~$\hat{f}_i$, $i\in\S$, were similarly plotted.  Thus, for
these 19 scenarios, the LWR \ctp\ is $47.8$ GW.  Because the analysis
of the 2015 ECR restricts the decision set~$\D$ as described earlier,
the \ctp\ recommended by that report is $47.9$ GW.  This is the \ctp\
which precisely meets the reliability standard for one of the other
(more minor) ``sensitivities'' considered in the report.  However, a
difference of $0.1$ GW is quite negligible in relation the possible
impact of many of the other uncertainties involved, for example the
true values of all the data underlying the analysis, or the values of
the parameters~\voll\ and \cone.

The lesson to be learned from the above analysis is that it is the two
``extreme'' scenarios or ``sensitivities'', as identified above, which
largely determine the LWR solution~$\lwr{x}_\S$.  This solution is
indifferent to the remaining scenarios so long as they do not
themselves become more extreme than either of the above two.  However,
the ``extreme'' scenarios are often, as here, relatively minor ones.
Considerable care nevertheless needs to be taken in their definition,
precisely because it is they that are critical in determining the
result of the analysis.

\paragraph{Alternative probabilistic analysis.}

In line with the alternative, Bayesian, approach outlined in
Section~\ref{sec:altern-prob-analys}, we consider briefly what
assignment of probabilities to scenarios would, on minimisation of the
expression given by~\eqref{eq:11} so as to determine the
solution~$x_{\mathrm{bayes}}$ of a Bayesian analysis, result in the
same value of the \ctp\ as that given by the solution~$\lwr{x}_\S$ of
the LWR analysis.  Since the cost functions~$f_i$ may here reasonably
be treated as differentiable, in addition to being convex, it follows
from~\eqref{eq:11} that the sets of such probabilities $p_i$,
$i\in\S$, are given by the solution of
\begin{equation}
  \label{eq:23}
  \sum_{i\in\S}p_if'_i(\lwr{x}_\S) = 0.
\end{equation}
If the set $\S$ is taken to consist of all 19 scenarios of the 2015
ECR, then the sets of probabilities such that \eqref{eq:23} holds form
a 17-dimensional region.  If we restrict attention to the two
``extreme'' sensitivities Slow Progression Low Demand and Consumer
Power Cold which determine the LWR solution~$\lwr{x}_\S$ in the sense
we have discussed in this paper, then some numerical experimentation
shows that the assignment of a probability $0.68$ to the first of
these and $0.32$ to the second (with a probability $0$ assigned to all
the remaining scenarios) gives $x_{\mathrm{bayes}}=\lwr{x}_\S=47.8$ GW
as required.  That these probabilities are approximately correct may
also be seen (using~\eqref{eq:23}) from
Figure~\ref{fig:ECR_regret_functions}.

A uniform assignment of probabilities~$0.2$ to each of the five major
scenarios considered by the 2015 ECR, and whose cost and regret
functions are plotted in Figures~\ref{fig:ecr_figure_14} and
\ref{fig:ECR_regret_functions}, results in a \emph{capacity to
  procure} of $x_{\mathrm{bayes}}=47.7$ GW.  The exclusion of the DECC
scenario (which, as previously explained, was not included in the
final analysis of the report) increases this to
$x_{\mathrm{bayes}}=48.0$ GW.  A uniform assignment of probabilities
to all 19 scenarios and ``sensitivities'' considered in the report
again results in a value $x_{\mathrm{bayes}}=48.0$ GW (and $48.1$ GW
if the DECC scenario is excluded).  As it happens, none of these
figures is here significantly different from the result of the LWR
analysis.  The latter is therefore reasonably robust, in this
particular case, across a range of plausible assignments of
probabilities to scenarios.  That this is so is essentially a
combination of the fact that the regret functions are close to being
obtained from each other by simple shifts of their arguments as
described in Section~\ref{sec:linear-shifts} and of the fact that the
scenarios are relatively evenly spaced between their extremes.

\section{Conclusion}
\label{sec:conclusion}

Minimax and LWR analysis are commonly used for decision making
whenever it is difficult, or perhaps inappropriate, to attach
probabilities to possible future scenarios.  We have shown that, for
each of these two approaches and subject only to the convexity of the
cost functions involved, it is always the case that there exist two
``extreme'' scenarios whose costs determine the outcome of the
analysis in the sense we have made precise in
Section~\ref{sec:least-worst-regret}.  (The ``extreme'' scenarios need
not be the same for each of the two approaches.)  The results of
either analysis are therefore particularly sensitive to the cost
functions associated with the corresponding two ``extreme'' scenarios
(while being largely unaffected by those associated with the
remainder).  In effect the results are sensitive to the
\emph{definitions} of these two scenarios, and indeed to whether or
not they are even included in the specification of the problem.  Great
care is therefore required in applications to identify these scenarios
and to consider their reasonableness.

We have also considered the common situation in which, at least to a
good approximation, the regret functions differ from each other
essentially by shifts of their arguments.  In this case the two
``extreme'' scenarios are the obvious ones, namely those whose regret
functions have the greatest relative shift in their arguments.  It is
further possible here to specify those sets of probabilities which, if
assigned to scenarios under an alternative Bayesian analysis, would
produce the answers as given by a LWR analysis in particular.  At a
minimum this assists in assessing the reasonableness of scenarios for
inclusion in a LWR analysis.

A particular example of the above situation may occur in the case of
determining a appropriate level of investment in order to control a
risk.  Here the cost functions associated with possible future
scenarios are often the sum of an exponentially decaying and a
linearly increasing component, and further the exponential decay rates
are often approximately the same for all scenarios.  This in the case
with the problem of determining an appropriate level of electricity
capacity procurement in Great Britain, where decisions must be made
several years in advance, in spite of considerable uncertainty as to
which of a number of future scenarios may occur, and where LWR
analysis is currently used as the basis of decision making.  We have
considered in detail the analysis of the 2015 Electricity Capacity
Report submitted to the UK Government by National Grid plc.  Here the
``extreme'' scenarios determining the result of the LWR analysis are
readily identified.  \emph{Given these}, the outcome of the LWR
analysis appears reasonably robust against a variety of assignments of
probabilities to the set of all the scenarios used.

\section*{Acknowledgements}
\label{sec:acknowledgements}

The author is grateful to National Grid plc for permission to use the
data upon which the analysis of Section~\ref{sec:exampl-electr-capac}
is based.  He is also grateful for the comments of a number of
colleagues, in particular Duncan Rimmer of National Grid and Chris
Dent and Amy Wilson of Durham University.

\end{document}